\documentclass[11pt]{amsart}

\usepackage[a4paper,margin=30mm]{geometry}
\usepackage[T1]{fontenc}
\usepackage{lmodern}
\usepackage{microtype}
\usepackage{amsmath,amssymb,amsthm,mathtools}
\usepackage{booktabs}
\usepackage{enumitem}
\usepackage{hyperref}

\hypersetup{
  colorlinks=true,
  linkcolor=blue,
  citecolor=blue,
  urlcolor=blue
}

\theoremstyle{plain}
\newtheorem{theorem}{Theorem}
\newtheorem{lemma}{Lemma}

\theoremstyle{remark}

\DeclareMathOperator{\Id}{Id}

\newcommand{\R}{\mathbb R}

\title[A simply connected nilpotent Lie group with a closed geodesic]
{A simply connected nilpotent Lie group with a closed geodesic}

\author{Glen Wheeler}
\address{School of Mathematics and Physics, University of Wollongong, Wollongong, NSW 2522, Australia}


\begin{document}

\begin{abstract}
We give a six-dimensional simply connected nilpotent Lie group together with a left-invariant Riemannian metric admitting a nonconstant closed geodesic.  The example gives a
positive answer to B\"ohm--Lafuente's question asking whether a homogeneous
Riemannian manifold diffeomorphic to Euclidean space can have a closed geodesic.
\end{abstract}

\maketitle

\section{Introduction}

B\"ohm and Lafuente asked whether a homogeneous Riemannian manifold diffeomorphic to Euclidean space can admit a nonconstant closed geodesic \cite[Section~8]{BoehmLafuente2018Immortal}.  They noted in particular that the question was open for left-invariant metrics on nilpotent Lie groups.

For nilpotent Lie groups, several nearby results point in the opposite direction.  Naitoh and Sakane studied geodesics and conjugate points on two-step nilpotent Lie groups \cite{NaitohSakane1981Conjugate}, and Berichon later formulated the B\"ohm--Lafuente problem explicitly and gave nonexistence results for simply connected two-step nilpotent groups with left-invariant metrics \cite{Berichon2019ClosedGeodesicsEuclidean}.  There is also a substantial literature on closed geodesics in compact nilmanifolds and, especially, compact quotients of two-step nilpotent Lie groups; see for instance \cite{DeCoste2008ClosedChevalley,Eberlein1994GeometryTwoStepI,Eberlein1994GeometryTwoStepII,GornetMast2003LengthMinimizing,LeePark1996SmoothlyClosed,Mast1994ClosedGeodesics}.  In that compact setting a geodesic may close modulo a lattice even when its lift to the simply connected nilpotent group does not close.  Homogeneous geodesics and geodesic orbit spaces form another related strand; see \cite{Gordon1996GeodesicOrbits,GordonNikonorov2018GOstructuresRn,KowalskiVanhecke1991HomogeneousGeodesics,Nikonorov2024GO_nilmanifolds}.  Recent work on closed geodesics and conjugate points in Lie groups and homogeneous spaces provides further context \cite{LeBrigantLichtenfelzPreston2024Conjugate}.

The purpose of this note is to give an explicit nilpotent example.  We prove the following theorem.

\begin{theorem}\label{thm:main}
There exists a six-dimensional simply connected nilpotent Lie group \(N\cong\R^6\) and a left-invariant Riemannian metric \(g\) on \(N\) admitting a nonconstant closed geodesic.  More precisely, there is a four-step nilpotent Lie algebra \(\mathfrak n\), an inner product on \(\mathfrak n\), and a curve
\[
  \gamma(t)=\exp(\cos t\,X+\sin t\,Y)
\]
in the corresponding simply connected Lie group such that \(\gamma\) is a nonconstant geodesic and \(\gamma(t+2\pi)=\gamma(t)\) for all \(t\).
\end{theorem}

The process that we undertook in order to locate the example is described in Section~\ref{sec:search}.
In Section~\ref{sec:example} we give the Lie algebra, the metric, and the closed geodesic, finishing the proof of Theorem~\ref{thm:main} hold.

\section{Searching for closed geodesics in homogeneous spaces}\label{sec:search}

Let \(N\) be a simply connected nilpotent Lie group with Lie algebra \(\mathfrak n\).  A left-invariant metric on \(N\) is determined by an inner product \(\langle\cdot,\cdot\rangle\) on \(\mathfrak n\).  If \(\gamma\colon I\to N\) is a smooth curve, write
\[
  u(t)=dL_{\gamma(t)^{-1}}\gamma'(t)\in\mathfrak n
\]
for its left logarithmic velocity.  The geodesic equation is the Euler--Arnold equation
\begin{equation}\label{eq:arnold-search}
  \langle u'(t),E\rangle=\langle [u(t),E],u(t)\rangle
  \qquad\text{for all }E\in\mathfrak n .
\end{equation}
This follows immediately from the Koszul formula for a left-invariant metric.

The search method is as follows.  Choose a derivation \(D\in\operatorname{Der}(\mathfrak n)\) such that
\[
  e^{TD}=\Id
\]
for some \(T>0\).  Then \(e^{tD}\) integrates to a periodic one-parameter group of automorphisms of \(N\).  If the inner product is chosen so that
\begin{equation}\label{eq:Dskew-search}
  \langle DA,B\rangle+\langle A,DB\rangle=0
  \qquad\text{for all }A,B\in\mathfrak n,
\end{equation}
then these automorphisms are isometries.  Hence, for any \(W\in\mathfrak n\), the curve
\begin{equation}\label{eq:gammaW}
  \gamma_W(t)=\exp(e^{tD}W)
\end{equation}
is automatically \(T\)-periodic.  This removes the reconstruction condition from the search: only the geodesic equation remains.

Suppose \(\gamma_W\) has left logarithmic velocity \(u(t)\).  Since \(D\) is a derivation,
\[
  u(t)=e^{tD}u_0,
  \qquad u_0=u(0).
\]
For example, if \(\mathfrak n\) has step at most four, then
\begin{equation}\label{eq:u0-general}
  u_0=DW-\frac12[W,DW]+\frac16[W,[W,DW]]-
  \frac1{24}[W,[W,[W,DW]]].
\end{equation}
Using \eqref{eq:Dskew-search} and the fact that \(e^{tD}\) is both an isometry and an automorphism, equation \eqref{eq:arnold-search} is equivalent to the single algebraic condition
\begin{equation}\label{eq:linear-geodesic-condition}
  \langle DE+[u_0,E],u_0\rangle=0
  \qquad\text{for all }E\in\mathfrak n .
\end{equation}
For fixed \((\mathfrak n,D,W)\), the vector \(u_0\) is known and \eqref{eq:linear-geodesic-condition}, together with \eqref{eq:Dskew-search}, is linear in the unknown inner product.  The remaining condition is positivity of the resulting Gram matrix.

This procedure naturally rules out some small trials.  On the Heisenberg algebra \([X,Y]=Z\), with \(DX=Y\), \(DY=-X\), \(DZ=0\), and \(W=X\), one obtains \(u_0=Y-\frac12Z\).  A \(D\)-invariant positive inner product makes \(\operatorname{span}\{X,Y\}\) orthogonal to \(\R Z\), and testing \eqref{eq:linear-geodesic-condition} against \(E=X\) gives a sum of positive terms.  A direct three-step extension with \([X,Z]=U\), \([Y,Z]=V\), and \(D(U)=V\), \(D(V)=-U\), still fails by a positivity obstruction.  Adding the fourth step terms \([X,U]\) and \([Y,V]\) supplies exactly the additional fixed direction needed to satisfy the same linear equations with a positive definite metric.  The resulting six-dimensional example is given below.

\section{The example}\label{sec:example}

Let \(\mathfrak n\) be the real vector space with ordered basis
\[
  (X,Y,Z,U,V,Q).
\]
Define a skew-symmetric bilinear bracket by declaring the following brackets to be nonzero:
\begin{equation}\label{eq:brackets1}
  [X,Y]=Z,
  \qquad [X,Z]=U,
  \qquad [Y,Z]=V,
\end{equation}
\begin{equation}\label{eq:brackets2}
  [X,U]=Q,
  \qquad [Y,V]=Q,
\end{equation}
and by setting all other brackets between basis vectors equal to zero, apart from those forced by skew-symmetry.

\begin{lemma}\label{lem:liealg}
Equations \eqref{eq:brackets1}--\eqref{eq:brackets2} define a four-step nilpotent Lie algebra.
\end{lemma}

\begin{proof}
Only the Jacobi identity requires verification.  Assign degrees
\[
  \deg X=\deg Y=1,
  \qquad \deg Z=2,
  \qquad \deg U=\deg V=3,
  \qquad \deg Q=4.
\]
Every listed nonzero bracket respects degree addition.  Since there is no basis vector of degree greater than four, any Jacobiator with total degree greater than four vanishes term by term.  The only triple of distinct basis vectors with total degree at most four is \((X,Y,Z)\), and
\[
  [X,[Y,Z]]+[Y,[Z,X]]+[Z,[X,Y]]
  =[X,V]+[Y,-U]+[Z,Z]=0.
\]
Thus the Jacobi identity holds.  The lower central series is
\[
  \mathfrak n_2=\operatorname{span}\{Z,U,V,Q\},\qquad
  \mathfrak n_3=\operatorname{span}\{U,V,Q\},
\]
\[
  \mathfrak n_4=\operatorname{span}\{Q\},\qquad
  \mathfrak n_5=0.
\]
Since \(Q=[X,[X,[X,Y]]]\neq0\), the nilpotency step is four.
\end{proof}

Let \(N\) be the simply connected Lie group with Lie algebra \(\mathfrak n\).  Since \(\mathfrak n\) is nilpotent, the exponential map \(\exp\colon\mathfrak n\to N\) is a global diffeomorphism.  Thus \(N\cong\mathfrak n\cong\R^6\).

Define an inner product on \(\mathfrak n\) by the Gram matrix
\begin{equation}\label{eq:metric}
G=
\begin{pmatrix}
5&0&0&0&24&0\\
0&5&0&-24&0&0\\
0&0&1&0&0&0\\
0&-24&0&117&0&0\\
24&0&0&0&117&0\\
0&0&0&0&0&108
\end{pmatrix}
\end{equation}
in the ordered basis \((X,Y,Z,U,V,Q)\).  After reordering the basis as \((X,V,Y,U,Z,Q)\), the matrix is block diagonal with blocks
\[
  \begin{pmatrix}5&24\\24&117\end{pmatrix},
  \qquad
  \begin{pmatrix}5&-24\\-24&117\end{pmatrix},
  \qquad (1),
  \qquad (108).
\]
The two nontrivial blocks have determinant \(5\cdot117-24^2=9>0\).  Hence \(G\) is positive definite.  Let \(g\) be the corresponding left-invariant Riemannian metric on \(N\).

Next define a linear map \(D\colon\mathfrak n\to\mathfrak n\) by
\begin{equation}\label{eq:D}
  DX=Y,
  \qquad DY=-X,
  \qquad DZ=0,
\end{equation}
\begin{equation}\label{eq:D2}
  DU=V,
  \qquad DV=-U,
  \qquad DQ=0.
\end{equation}
A direct check on the five nonzero brackets shows that \(D\) is a derivation.  Indeed,
\[
  D[X,Y]=0=[Y,Y]+[X,-X],
\]
\[
  D[X,Z]=V=[Y,Z]+[X,0],
  \qquad
  D[Y,Z]=-U=[-X,Z]+[Y,0],
\]
\[
  D[X,U]=0=[Y,U]+[X,V],
  \qquad
  D[Y,V]=0=[-X,V]+[Y,-U].
\]
For the zero brackets, the only nontrivial cancellations are
\[
  0=D[X,V]=[Y,V]+[X,-U]=Q-Q
\]
and
\[
  0=D[Y,U]=[-X,U]+[Y,V]=-Q+Q.
\]
All other zero bracket checks are immediate.

The matrix of \(D\) in the ordered basis is
\[
D=
\begin{pmatrix}
0&-1&0&0&0&0\\
1&0&0&0&0&0\\
0&0&0&0&0&0\\
0&0&0&0&-1&0\\
0&0&0&1&0&0\\
0&0&0&0&0&0
\end{pmatrix}.
\]
A direct multiplication gives
\begin{equation}\label{eq:D-skew}
  D^TG+GD=0.
\end{equation}
Thus \(D\) is skew-symmetric for \(\langle\cdot,\cdot\rangle\), and \(e^{tD}\) acts by isometric automorphisms of \((N,g)\).  Also \(e^{2\pi D}=\Id\).

Now set
\begin{equation}\label{eq:gamma}
  A(t)=\cos t\,X+\sin t\,Y,
  \qquad
  \gamma(t)=\exp(A(t)).
\end{equation}
Since the exponential map is a diffeomorphism and \(A(t+2\pi)=A(t)\), \(\gamma\) is \(2\pi\)-periodic.  It is nonconstant because \(A(0)=X\) and \(A(\pi)=-X\).  Equivalently,
\[
  \gamma(t)=\exp(e^{tD}X),
\]
so \(\gamma\) is an orbit of the periodic isometric automorphism group generated by \(D\).

It remains to verify that \(\gamma\) is a geodesic.  Since \(\mathfrak n\) is step four, the left logarithmic velocity of \(\gamma(t)=\exp(A(t))\) is
\begin{equation}\label{eq:left-log-exp}
  u(t)=A'(t)-\frac12[A(t),A'(t)]
  +\frac16[A(t),[A(t),A'(t)]]
  -\frac1{24}[A(t),[A(t),[A(t),A'(t)]]].
\end{equation}
Here \(A'(t)=-\sin t\,X+\cos t\,Y\), and the relevant brackets are
\[
  [A(t),A'(t)]=Z,
  \qquad
  [A(t),Z]=\cos t\,U+\sin t\,V,
\]
\[
  [A(t),[A(t),Z]]=Q.
\]
Therefore
\begin{equation}\label{eq:u}
  u(t)=-\sin t\,X+
       \cos t\,Y-
       \frac12 Z+
       \frac16(\cos t\,U+
       \sin t\,V)-
       \frac1{24}Q.
\end{equation}
In particular,
\begin{equation}\label{eq:u0}
  u_0:=u(0)=Y-\frac12 Z+\frac16U-\frac1{24}Q,
\end{equation}
and \(u(t)=e^{tD}u_0\).  Also \(\langle u_0,u_0\rangle=11/16\), so the curve has nonzero speed.

Because \(u(t)=e^{tD}u_0\), and because \(e^{tD}\) is both an isometry and an automorphism, the Euler--Arnold equation \eqref{eq:arnold-search} is equivalent to checking
\begin{equation}\label{eq:reduced-example}
  \langle DE+[u_0,E],u_0\rangle=0
  \qquad\text{for every }E\in\mathfrak n.
\end{equation}
We have
\begin{equation}\label{eq:Gu0}
  Gu_0=
  \left(0,\,1,\,-\frac12,\,-\frac92,\,0,\,-\frac92\right)^T,
\end{equation}
so
\[
  \langle X,u_0\rangle=0,
  \quad
  \langle Y,u_0\rangle=1,
  \quad
  \langle Z,u_0\rangle=-\frac12,
\]
\[
  \langle U,u_0\rangle=-\frac92,
  \quad
  \langle V,u_0\rangle=0,
  \quad
  \langle Q,u_0\rangle=-\frac92.
\]
Using the bracket table, one obtains
\[
\begin{array}{c|c|c}
E & DE+[u_0,E] & \langle DE+[u_0,E],u_0\rangle\\
\midrule
X & Y-Z+\frac12U-\frac16Q
  & 1+\frac12-\frac94+\frac34=0\\[1mm]
Y & -X+\frac12V
  & 0\\[1mm]
Z & V
  & 0\\[1mm]
U & V
  & 0\\[1mm]
V & -U+Q
  & \frac92-\frac92=0\\[1mm]
Q & 0
  & 0.
\end{array}
\]
Thus \eqref{eq:reduced-example} holds on a basis, hence for all \(E\in\mathfrak n\).  Therefore \(\gamma\) is a geodesic.  We have proved that \((N,g)\) is a homogeneous Riemannian manifold diffeomorphic to \(\R^6\) and that it admits the nonconstant closed geodesic \eqref{eq:gamma}.  This proves Theorem~\ref{thm:main}.

\section*{Acknowledgements}

The author acknowledges partial financial support from his current Australian Research Council (ARC) grants DP250101080 and  FT250100880.  He thanks Max Orchard for introducing him to this problem, and Max Orchard, James Stanfield, and Ramiro Lafuente for their interest.

\bibliographystyle{plain}
\bibliography{blbib}

@article{BoehmLafuente2018Immortal,
  author        = {B{\"o}hm, Christoph and Lafuente, Ramiro A.},
  title         = {Immortal homogeneous Ricci flows},
  journal       = {Inventiones Mathematicae},
  volume        = {212},
  number        = {2},
  pages         = {461--529},
  year          = {2018},
  doi           = {10.1007/s00222-017-0771-z},
  eprint        = {1701.00628},
  archivePrefix = {arXiv},
  primaryClass  = {math.DG}
}

@misc{Berichon2019ClosedGeodesicsEuclidean,
  author       = {Berichon, Rohin},
  title        = {Closed Geodesics on Euclidean Homogeneous Spaces},
  howpublished = {AMSI Vacation Research Scholarship report, University of Queensland},
  year         = {2019},
  url          = {https://amsi.org.au/2018-19-vrs/closed-geodesics-on-euclidean-homogeneous-spaces/}
}

@article{NaitohSakane1981Conjugate,
  author  = {Naitoh, Hiroo and Sakane, Yusuke},
  title   = {On conjugate points of a nilpotent Lie group},
  journal = {Tsukuba Journal of Mathematics},
  volume  = {5},
  number  = {1},
  pages   = {143--152},
  year    = {1981},
  doi     = {10.21099/tkbjm/1496159325}
}

@article{Eberlein1994GeometryTwoStepI,
  author  = {Eberlein, Patrick},
  title   = {Geometry of {$2$}-step nilpotent groups with a left invariant metric},
  journal = {Annales scientifiques de l'{\'E}cole Normale Sup{\'e}rieure},
  series  = {4},
  volume  = {27},
  number  = {5},
  pages   = {611--660},
  year    = {1994},
  doi     = {10.24033/asens.1702}
}

@article{Eberlein1994GeometryTwoStepII,
  author  = {Eberlein, Patrick},
  title   = {Geometry of {$2$}-step nilpotent groups with a left invariant metric. {II}},
  journal = {Transactions of the American Mathematical Society},
  volume  = {343},
  number  = {2},
  pages   = {805--828},
  year    = {1994},
  doi     = {10.1090/S0002-9947-1994-1214782-2}
}

@article{Mast1994ClosedGeodesics,
  author  = {Mast, Maura B.},
  title   = {Closed Geodesics in {$2$}-Step Nilmanifolds},
  journal = {Indiana University Mathematics Journal},
  volume  = {43},
  number  = {3},
  pages   = {885--911},
  year    = {1994},
  doi     = {10.1512/iumj.1994.43.43038}
}

@article{LeePark1996SmoothlyClosed,
  author  = {Lee, Kyung Bai and Park, Keun},
  title   = {Smoothly closed geodesics in {$2$}-step nilmanifolds},
  journal = {Indiana University Mathematics Journal},
  volume  = {45},
  number  = {1},
  pages   = {1--14},
  year    = {1996},
  doi     = {10.1512/iumj.1996.45.1077}
}

@article{GornetMast2003LengthMinimizing,
  author  = {Gornet, Ruth and Mast, Maura B.},
  title   = {Length minimizing geodesics and the length spectrum of {Riemannian} two-step nilmanifolds},
  journal = {The Journal of Geometric Analysis},
  volume  = {13},
  number  = {1},
  pages   = {107--143},
  year    = {2003},
  doi     = {10.1007/BF02931000}
}

@article{DeCoste2008ClosedChevalley,
  author        = {DeCoste, Rachelle C.},
  title         = {Closed geodesics on compact nilmanifolds with {Chevalley} rational structure},
  journal       = {Manuscripta Mathematica},
  volume        = {127},
  number        = {3},
  pages         = {309--343},
  year          = {2008},
  doi           = {10.1007/s00229-008-0206-7},
  eprint        = {math/0703930},
  archivePrefix = {arXiv}
}

@article{KowalskiVanhecke1991HomogeneousGeodesics,
  author  = {Kowalski, Old{\v r}ich and Vanhecke, Lieven},
  title   = {Riemannian manifolds with homogeneous geodesics},
  journal = {Bollettino dell'Unione Matematica Italiana. B},
  volume  = {5},
  pages   = {189--246},
  year    = {1991}
}

@incollection{Gordon1996GeodesicOrbits,
  author    = {Gordon, Carolyn S.},
  title     = {Homogeneous {Riemannian} manifolds whose geodesics are orbits},
  booktitle = {Topics in Geometry: In Memory of Joseph D'Atri},
  series    = {Progress in Nonlinear Differential Equations and Their Applications},
  volume    = {20},
  pages     = {155--174},
  publisher = {Birkh{\"a}user},
  address   = {Boston},
  year      = {1996}
}

@article{GordonNikonorov2018GOstructuresRn,
  author        = {Gordon, Carolyn S. and Nikonorov, Yurii G.},
  title         = {Geodesic orbit {Riemannian} structures on {$\mathbb{R}^n$}},
  journal       = {Journal of Geometry and Physics},
  volume        = {134},
  pages         = {235--243},
  year          = {2018},
  doi           = {10.1016/j.geomphys.2018.08.018},
  eprint        = {1803.01023},
  archivePrefix = {arXiv},
  primaryClass  = {math.DG}
}

@article{Nikonorov2024GO_nilmanifolds,
  author        = {Nikonorov, Yurii G.},
  title         = {On geodesic orbit nilmanifolds},
  journal       = {Journal of Geometry and Physics},
  volume        = {203},
  pages         = {105257},
  year          = {2024},
  doi           = {10.1016/j.geomphys.2024.105257},
  eprint        = {2402.17548},
  archivePrefix = {arXiv},
  primaryClass  = {math.DG}
}

@misc{LeBrigantLichtenfelzPreston2024Conjugate,
  author        = {Le Brigant, Alice and Lichtenfelz, Leandro and Preston, Stephen C.},
  title         = {Conjugate points on {Lie} groups with left-invariant metrics},
  year          = {2024},
  eprint        = {2408.03854},
  archivePrefix = {arXiv},
  primaryClass  = {math.DG},
  note          = {Version 3 appeared in 2025}
}

\end{document}